\documentclass{amsproc}
\pdfoutput=1
\usepackage{amsmath, amsthm, amssymb}
\usepackage{graphicx}
\usepackage{lscape}
\usepackage{color}
\newtheorem{theorem}{Theorem}[section]
\newtheorem{lemma}{Lemma}[section]

\newtheorem{corollary}{Corollary}[section]

\newtheorem{example}{Example}[section]
\newcommand{\K}{\mathcal{K}}

\theoremstyle{definition}

\def\mk{\mathcal{K}}

\def\p{\prime}

\numberwithin{equation}{section}

\begin{document}

\title[Nullification of Knots and Links]{Nullification of knots and links}

\author{Yuanan Diao}
\address{Department of Mathematics\\ University of North Carolina at Charlotte\\
Charlotte, NC
28223, USA} \email{ydiao@uncc.edu}

\author{Claus Ernst}
\address{Department of Mathematics and Computer Science,
Western Kentucky University, Bowling Green, KY 42101, USA}
\email{claus.ernst@wku.edu}

\author{Anthony Montemayor}
\address{Department of Mathematics and Computer Science,
Western Kentucky University, Bowling Green, KY 42101, USA}
\email{anthony.montemayor477@wku.edu}

\thanks{Y. Diao is
currently supported by NSF Grants \#DMS-0920880 and \#DMS-1016460, C. Ernst  was
partially supported  by an internal summer grant from WKU in 2010 and is currently supported NSF grant \#DMS-1016420.}


\begin{abstract}
In this paper, we study a geometric/topological measure of knots
and links called the {\em nullification number}. The nullification of knots/links is
believed to be biologically relevant. For example, in DNA topology, one can intuitively regard it as a way to measure how easily a knotted circular DNA can unknot itself through recombination of its DNA strands. It turns out
that there are several different ways to define such a number.
These definitions lead to nullification
numbers that are related, but different. Our aim is to explore the mathematical
properties of these nullification numbers. First, we give specific examples to show that the nullification numbers we defined are different. We provide detailed analysis of the nullification numbers for the well known 2-bridge knots and links. We also explore the relationships among the three nullification numbers, as well as their relationships with other knot invariants. Finally, we study a special class of links, namely those links whose general nullification number equals one.  We show that such links exist in abundance. In fact, the number of such links with crossing number less than or equal to $n$ grows exponentially with respect to $n$.
\end{abstract}

\maketitle

\section{Introduction and basic concepts}

\medskip
Historically, knots and links had been a subject of study almost
only in the  realm of pure mathematics. That has dramatically
changed since the discovery of knots and links formed in circular
DNA a few decades ago. It turned out that the topology of the
circular DNA plays a very  important role in the properties of
the DNA. Various geometric and topological complexity measures of
DNA knots that are believed to be biologically relevant, such as
the knot types, the 3D writhe, the average crossing numbers, the
average radius of gyration, have been studied. In this paper, we
are interested in another geometric/topological measure of knots
and links called the {\em nullification number}, which is also
believed to be biologically relevant \cite{BuMa, GrBrVaSiSh}. Intuitively, this
number measures how easily a knotted circular DNA can unknot
itself through recombination of its DNA strands. It turns out
that there are several different ways to define such a number.
These different definitions lead to different nullification
numbers that are related. Our aim is to explore the mathematical
properties of the nullification numbers. In this section, we will
outline a brief introduction to basic knot theory concepts. In
Section 2, we will give precise definitions for three different
nullification numbers. In Section 3, we will study the
nullification numbers for a well known class of knots called the class of
Montesinos knots and links. In Section 4, we explore the relationships among the three nullification numbers, as well as their relationships with other knot  invariants. In particular, we give examples to show that the three nullification numbers defined here are indeed different. In Section 5, we study a special class of links, namely the links whose general nullification number equals one.  There we show that such links exist in abundance. In fact, the number of such links with crossing number less than or equal to $n$ grows exponentially with respect to $n$.

\medskip
Let $K$ be a tame link, that is, $K$ is a collection of several
piece-wise smooth simple closed curves in $\mathbb{R}^3$. In the
particular case that $K$ contains only one component, it is
called a knot instead. However through out this paper a link
always includes the special case that it may be a knot, unless
otherwise stated. A link is oriented if each component of the
link has an orientation. Intuitively, if one can continuously
deform a tame link $K_1$ to another tame link $K_2$ (in
$\mathbb{R}^3$), then $K_1$ and $K_2$ are considered  equivalent
links in the topological sense. The corresponding continuous
deformation is called an {\it ambient isotopy}, and $K_1$, $K_2$
are said to be ambient isotopic to each other. The set of all
(tame) links that are ambient isotopic to each other is called a
\textit{link type}. For a fixed link (type) $\mk$, a link diagram
of $\mk$ is a projection of a member $K\in\mk$ onto a plane. Such
a projection $p:\ K\subset R^3 \rightarrow D \subset R^2$ is
regular if the set of points $\lbrace x \in D :\vert
p^{-1}(x)\vert
>1\rbrace$ is finite and there is no $x$ in $D$ for which $\vert
p^{-1}(x)\vert
>2$. In other words, in the diagram no more than two arcs of $D$
cross at any point in the projection and there are only finitely
many points where the arcs cross each other. A point where two
arcs of $D$ cross each other is called a crossing point, or just
a crossing of $D$. The number of crossings in $D$ not only
depends on the link type $\mk$, it also depends on the
geometrical shape of the member $K$ representing $\mk$ and the
projection direction chosen. The minimum number of crossings in
all regular projections of all members of $\mk$ is called the
{\it crossing number} of the link type $\mk$ and is denoted by
$Cr(\mk)$. For any member $K$ of $\mk$, we also write
$Cr(K)=Cr(\mk)$. Of course, by this definition, if $K_1$ and
$K_2$ are of the same link type, then we have $Cr(K_1)=Cr(K_2)$.
However, it may be the case that for a member $K$ of $\mk$, none
of the regular projections of $K$ has crossing number $Cr(\mk)$.
A diagram $D$ of a link $K\in\mk$ is {\it minimum} if the number
of crossings in the diagram equals $Cr(\mk)$. We will often call
$D$ a {\it minimum projection diagram}. A link diagram is {\it
alternating} if one encounters over-passes and under-passes
alternatingly when traveling along the link projection. A diagram
$D$ is said to be {\em reducible} if there exists a crossing
point in $D$ such that removing this crossing point makes the
remaining diagram two disconnected parts. $D$ is {\em reduced} if
it is not reducible. A link is {\em alternating} if it has a
reduced alternating diagram. A famous result derived from the
Jones polynomial is that the crossing number of an alternating
link $\mk$ equals the number of crossings in any of its reduced
alternating diagram since each diagram is minimum. For example
the diagram of the knot $\mk_1$ in Figure \ref{composite} is
minimum.

\medskip
A link $\mk$ is called a {\it composite link} if a member of it
can be obtained by cutting open two nontrivial links $K_1$ and
$K_2$ and reconnecting the strings as shown in Figure
\ref{composite}. The resulting link is written as $K=K_1\#K_2$
and $K_1$, $K_2$ are called the {\it connected sum components} of
$K$. Of course a link can have  more than two connected sum
components. A link $\mk$ that is not a composite link is called a
{\it prime link}.

\begin{figure}[h!]
\begin{center}
\includegraphics[scale=0.6]{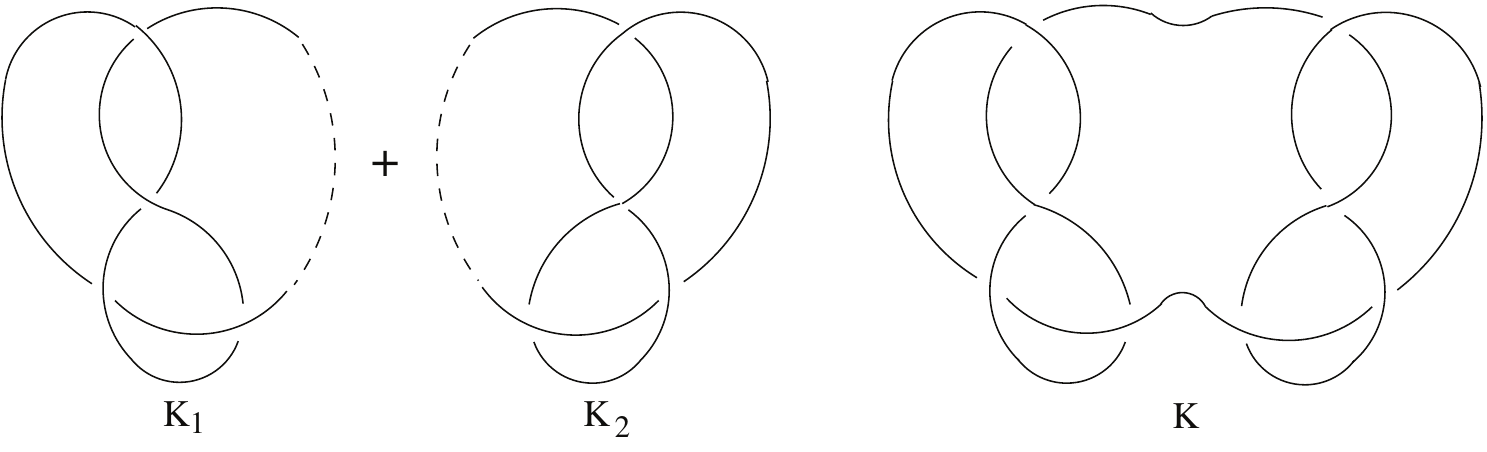}
\caption{\label{composite} A composite knot $K=K_1\#K_2$.}
\end{center}
\end{figure}

In the case of alternating links, any two
minimum projection diagrams $D$ and $D'$ of the same alternating
link $\mk$ are flype equivalent, that is, $D$ can be changed to
$D'$ through a finite sequence of flypes \cite{MT1, MT2} (see
Figure \ref{flype}).

\begin{figure}[h!]
\begin{center}
\includegraphics[scale=1.0]{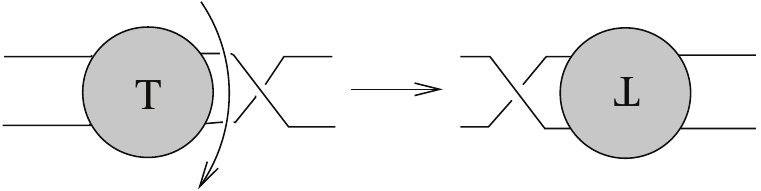}
\caption{\label{flype} A single flype. $T$ denotes a part of the
diagram that is rotated by 180 degrees by the flype.}
\end{center}
\end{figure}

Let $c$ be a crossing in an alternating diagram $D$. The {\em flyping
circuit of $c$} is defined as the unique decomposition of $D$ into
crossings $c_1, c_2, \ldots, c_m$, $m\ge 1$ and tangle diagrams
$T_1, T_2, \ldots, T_r$, $r\ge 0$ joined together as shown in Figure
\ref{flypingcircuit} such that (i) $c=c_i$ for some $i$ and (ii) the
$T_i$ are minimum with respect to the pattern.

\begin{figure}[h!]
\begin{center}
\includegraphics[scale=.6]{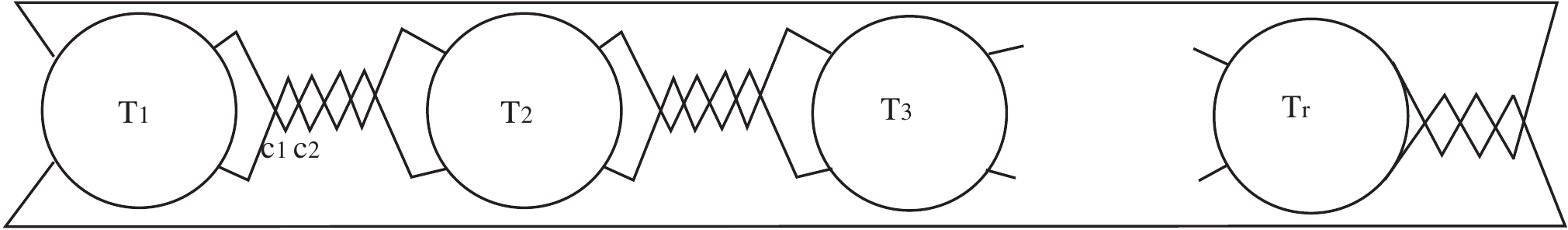}
\caption{\label{flypingcircuit} A flyping circuit. Any crossing in the
flyping circuit can be flyped to any position between tangles $T_i$
and $T_{i+1}$.}
\end{center}
\end{figure}

\medskip
\section{Definitions of Nullification Numbers}

\medskip
Let $D$ be a regular diagram of an oriented link $\mk$. A
crossing in $D$  is said to be {\em smoothed} if the strands of
$D$ at the crossing are cut and re-connected as shown in Figure
\ref{smooth}. If every crossing in $D$ is smoothed, the result
will be a collection of disjoint (topological) circles without
self intersections. These are called the {\em Seifert circles} of
$D$. Of course the set of Seifert circles of $D$ represents a
trivial link diagram. However, it is not necessary to smooth
every crossing of $D$ to make it a trivial link diagram. For
example, if a diagram has only one crossing, or only two crossing
with only one component, then the diagram is already a trivial
link diagram. So the minimum number of crossings needed to be
smoothed in order to turn $D$ into a trivial link diagram is
strictly less than the number of crossings in $D$. This minimum
number is called the {\em nullification number} of the diagram
$D$, which we will write as $n_D$. Notice that $n_D$ is not a
link invariant since different diagrams (of the same link) may
have different nullification numbers. In order to define a number
that is a link invariant, we would have to consider the set of
all diagrams of a link. Depending on how we choose to smooth the
crossings in the process, we may then get different versions of
nullification numbers. This approach is in a way similar to the
how different versions of unknotting numbers are defined in
\cite{DES}.

\begin{figure}[h!]
\begin{center}
\includegraphics[scale=.4]{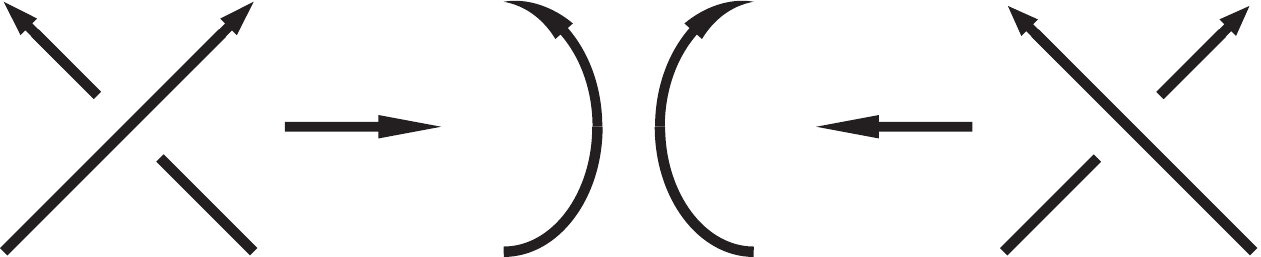}
\caption{\label{smooth} The smoothing of a single crossing.}
\end{center}
\end{figure}

\medskip
Let $\mk$ be an oriented link and $D$ be a (regular) diagram of
$\mk$.  Choose some crossings in $D$ and smooth them. This
results in a new diagram $D^\prime$ which is most likely of a
different link type other than $\mk$. Suppose we are allowed to
deform $D^\prime$ (without changing its link type, of course) to
a new diagram $D_1$. We can then again choose some crossings in
$D_1$ to smooth and repeat this process. With proper choices of
the new diagrams and the crossings to be smoothed, it is easy to
see that this process can always terminate into a trivial link
diagram. The minimum number of crossings required to be smoothed
in order to make any diagram $D$ of $\mk$ into a trivial link
diagram by the above procedure is then defined as the {\em
general nullification number} of $\mk$, or just the nullification
number of $\mk$. We will denote it by $n(\mk)$.

\medskip
On the other hand, if in the above nullification procedure, we
require that the diagrams used at each step be minimum diagrams
(of their corresponding link types), then the minimum number of
crossings required to be smoothed in order to make any minimum
diagram $D$ of $\mk$ into a trivial link diagram is defined as
the {\em restricted nullification number} of $\mk$, which we will
denote by $n_r(\K)$.

\medskip
In the case that $D$ is a minimum diagram of $\mk$, we have
already defined the nullification  number $n_D$ for the diagram
$D$, namely the minimum number of smoothing moves needed to
change $D$ into a trivial link diagram. If we take the minimum of
$n_D$ over all minimum diagrams of $\mk$, then we obtain a third
nullification number of $\mk$, which we will call the {\em
diagram nullification number} of $\mk$ and will denoted it by
$n_d(\mk)$.

\medskip
By the above definitions, clearly we have
$$
n(K)\le n_r(\mk) \le n_d(\mk).
$$

We shall see later that these definitions of nullification
numbers are indeed all different. Of the three nullification
numbers, the diagram nullification number $n_d(\mk)$ has been
studied in \cite{C, S}.   Specifically, in \cite{S} it is shown
that for an alternating link $\mk$, the diagram nullification
number $n_d(\mk)$ can be computed from any reduced alternating
diagram $D$ of $\mk$ using the following formula:
\begin{equation}
\label{solaeq}
n_d(\mk)=Cr(D)-s(D) +1,
\end{equation}
where $s(D)$ is the number of Seifert circles in $D$. This allows
us to express the genus $g(\mk)$ of an alternating link $\mk$ in
terms of the number of Seifert circles and the nullification
number by
\begin{equation}
\label{genuseq}
g(\mk) =\frac{1}{2} (n_d(\mk)-\nu+1),
\end{equation}
where $\nu$ is the number of components of $\mk$. On the other
hand, the diagram nullification number $n_d$ for alternating
links is closely related to the HOMFLY polynomial by  the
following lemma. This provides an expression of $n_d$ without
having to make reference to a particular diagram.

\medskip
\begin{lemma}\label{nDrelationtopoly}
Let $\mk$ be an alternating non-split link, then $
n_D(\mk)=\beta_z, $ where $\beta_z$ is the maximum degree of the
variable $z$ in the HOMFLY polynomial $P_\mk(v,z)$ of $\mk$.
\end{lemma}

\medskip The result of Lemma \ref{nDrelationtopoly} is given in \cite{S}.
In the following we give a short proof of the lemma, since a
proof was not given in \cite{S}.  For more and detailed
information regarding HOMFLY polynomial and other facts in knot
theory, please refer to a standard text in knot theory such as
\cite{Crom}.

\medskip
\begin{proof}
For any link $\mk$ with diagram $D$ we have the inequality $
\beta_z\le Cr(D)-s(D)+1. $ The Conway polynomial $C_\mk(z)$ of
$\mk$ is related to $P_\mk(v,z)$ by  the equation
$C_\mk(z)=P_\mk(1,z)$. It follows that the maximum degree
$\alpha_z$ of $z$ in $C_\mk(z)$ is at most $\beta_z$. Since $\mk$
is alternating, $\alpha_z=2g(\mk)+\mu-1$, where $g(\mk)$ is the
genus of $\mk$ and $\mu$ is the number of components of $\mk$.
Using the Seifert algorithm on a reduced alternating diagram of
$\mk$ we get
$$
g(\mk)=\frac{1}{2}(Cr(D)-s(D)-\mu)+1.
$$
Now it follows that
\begin{eqnarray*}
\alpha_z&=& 2g(\mk)+\mu-1\\
&=& (Cr(D)-s(D)-\mu)+\mu+1\\
&=&Cr(D)-s(D)+1\\
&\ge & \beta_z.
\end{eqnarray*}
Thus $\beta_z=\alpha_z=Cr(D)-s(D)+1=n_D(\mk)$.
\end{proof}

\medskip
In general, if $D$ is a diagram of some non-alternating link,
then we have $n_D\le Cr(D)-s(D)+1$, but the precise determination
of $n_D$ (hence $n_d(\mk)$) is far more difficult. In the
following we propose a different inequality concerning $n_D$
using the concept of parallel and anti-parallel crossings. A
flyping circuit is said to be {\em nontrivial} if it either
contains more than one crossing or more than one tangle.
Otherwise it is called a {\em trivial} flyping circuit.  If a
crossing is part of a nontrivial flyping circuit then the
crossing belongs to a unique flyping circuit \cite{calvo}. If
there is more than one crossing in a flyping circuit of an
oriented link diagram, then the crossings in the flyping circuit
are called {\em parallel} or {\em anti-parallel} as shown in
Figure \ref{parandntipar}.

\begin{figure}[h!]
\begin{center}
\includegraphics[scale=.4]{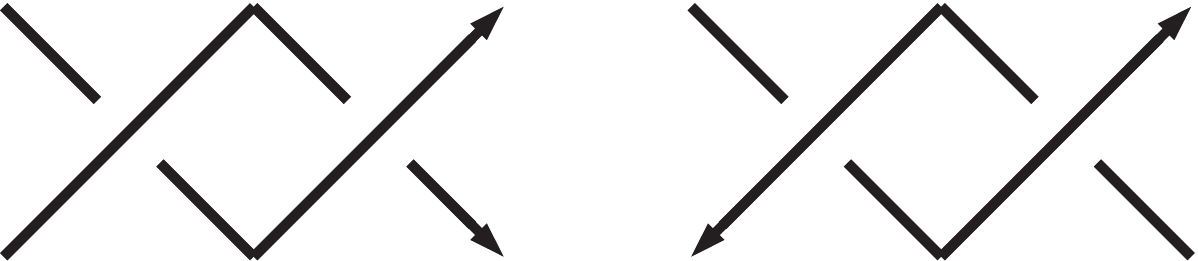}
\caption{\label{parandntipar} Parallel (left) and anti-parallel
(right)  crossings in a nontrivial flyping circuit.}
\end{center}
\end{figure}

Note that we can assign the notion of parallel or anti-parallel
even to a  single crossing as long as the flyping circuit has at
least two tangles. For trivial flyping circuits that consist of a
single crossing and a single tangle there is no obvious way of
assigning a notion of parallel or anti-parallel to it. A
nontrivial flyping circuit has the special property that all the
crossings in it can be eliminated by nullifying (i.e. smoothing)
a single crossing if the crossings are anti-parallel, while all
crossings in it have to be smoothed (in order to eliminate them
with nullifying moves within the circuit) when the crossings are
parallel. Let $P_1$, $P_2$, ..., $P_m$ be the nontrivial flyping
circuits with parallel crossings and let $|P_i|$ be the number of
crossings in $P_i$. Let $A$ be number of nontrivial flyping
circuits with anti-parallel crossings and $S$ be the total number
of crossings in all trivial flyping circuits. We conjecture that
for any link diagram $D$
\begin{equation}
\label{ineqfornD}
n_D\le \sum_{1\le i\le m}(|P_i|-1)+A+S+c,
\end{equation}
where $c\le 1$ is an additional constant depending on the link
type of $D$.   For alternating diagrams, we expect an almost
equality in (\ref{ineqfornD}), while for non-alternating diagrams
(\ref{ineqfornD}) may still be a large overestimate. Figure
\ref{mathknot11a263} shows the case of a minimum diagram for the
knot $11a_{263}$. There are four visible nontrivial flyping
circuits (three of which have 3 crossings and one with two
crossings) and all crossings in the circuits are parallel and
each crossing belong to one such circuit. It follows that $A=S=0$
and $\sum_{1\le i\le 4}(|P_i|-1)+A+S=7$. Thus (\ref{ineqfornD})
becomes an equality with the choice of $c=1$ since $n_D=8$. This
example shows that it is necessary for us to have the constant
$c$ term in general.

\begin{figure}[h!]
\begin{center}
\includegraphics[scale=.4]{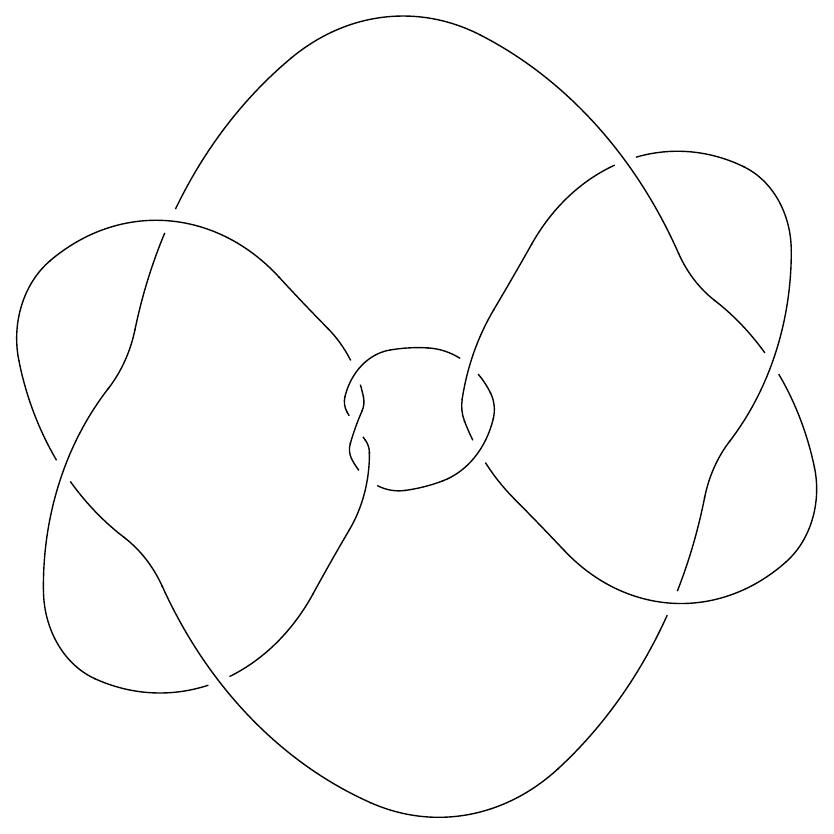}
\caption{\label{mathknot11a263} A minimum diagram knot
$11a_{263}$ with $n_D=8$.  It has four parallel nontrivial
flyping circuits and each crossing belongs to one of these
circuits.}
\end{center}
\end{figure}

\medskip
\section{Diagram Nullification Numbers of 4-plats and Montesinos Links}

\medskip
In this section we discuss the nullification number $n_D$ of
4-plats and Montesinos links.  The goal is to show that the
inequality in (\ref{ineqfornD}) holds for these links.

\medskip
A 4-plat is a link with up to two components that admits a
minimum alternating diagram  as shown in Figure \ref{4plat} where
a grey box marked by $c_i$ indicates a row of $c_i$ horizontal half-twists. Such a link is
completely defined by such a vector $(c_1,c_2,\ldots,c_k)$ of positive integer entries.
Obviously, two vectors of the form $(c_1,c_2,\ldots,c_k)$ and
$(c_k,c_{k-1},\ldots,c_1)$ define the same link. However, it is much less obvious that two such vectors define
different 4-plats if they are not reversal of each other. For a detailed discussion on the
classification of 4-plats see \cite{BZ, Crom}. In a standard
4-plat diagram there is an obvious way to assign the notion of
parallel or anti-parallel to a single crossing, based on if both
strings move in the same right-left direction.

\begin{figure}[h!]
\begin{center}
\includegraphics[scale=0.9]{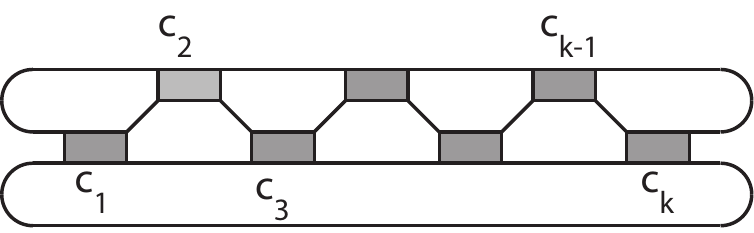}
\caption{\label{4plat} A typical 4-plat template. A gray box with label $c_i$
represents a horizontal sequence of $c_i$ crossings.}
\end{center}
\end{figure}

\medskip
A similar schema based on a vector $T = (a_1,a_2,\ldots,a_n)$ is
used to classify rational tangles. A tangle $T$ is part of a
link diagram that consists of a disk that contains two properly
embedded arcs. For a typical rational tangle diagram see Figure
\ref{tanglediagram} where the rectangular box contains either
horizontal or vertical half-twists. A horizontal (vertical)
rectangle labeled $a_i$ contains $|a_i|$ horizontal (vertical)
half-twists and horizontal and vertical rectangles occur in an
alternating fashion. All rational tangles end with $a_n$
horizontal twists on the right. The $a_i$'s are either all
positive or all negative, with the only exception that $a_n$ may
equal to zero. For a classification and precise definition of
such tangles see \cite{BZ, Crom}. We assign the notion of
parallel or anti-parallel to a single crossing, based on if both
strings move in the same right-left direction for a horizontal
crossing and based on if both strings move in the same up-down
direction for a vertical crossing.

\begin{figure}[h!]
\begin{center}
\includegraphics[scale=.6]{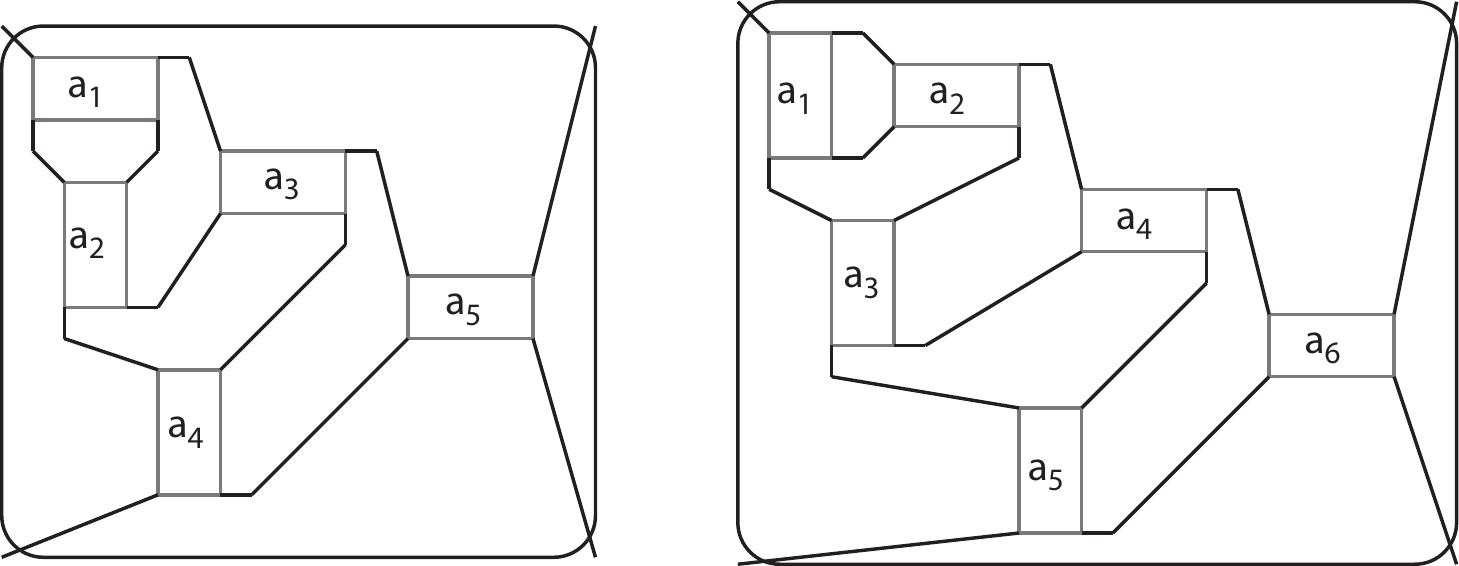}
\caption{A rational tangle diagram $T$ given by the vector
$T=(a_1,a_2,\ldots,a_n)$. On the left is a  diagram with $n$ odd
($n=5$) and on the left is a diagram with $n$ even ($n=6$). A
small rectangle with label $a_i$ contains $|a_i|$ half-twists,
where $a_i$'s are either all positive or all negative, with the
only exception that $a_n$ may equal to zero. }
\label{tanglediagram}
\end{center}
\end{figure}

\medskip
Each rational tangle $T=(a_1,a_2,\ldots,a_n)$ defines a rational
number $\beta/\alpha$ using the continued fraction expansion:
$$
\frac{\beta}{\alpha}=a_n+\frac{1}{a_{n-1}+\frac{1}{a_{n-1}+\cdots +\frac{1}{a_1}}}.
$$

\medskip
Rational tangles are the basic building blocks of a large family
of links called Montesinos links. A Montesinos link admits  a
diagram that consists of rational tangles $T_i$ strung together
as shown in Figure \ref{montesionsLink} together with a
horizontal number of $|e|$ half-twists (as indicated by the
rectangle in the figure). Such a diagram is called a {\em
Montesinos diagram}. We say that a Montesinos diagram is of type
I if the orientations between the two arcs connecting any two
adjacent tangles in the diagram are parallel. In this case the
orientations between the two arcs connecting any two other
adjacent tangles in the diagram must be parallel as well. We say
that diagram is of type II otherwise.

\begin{figure}[h!]
\begin{center}
\includegraphics[scale=.4]{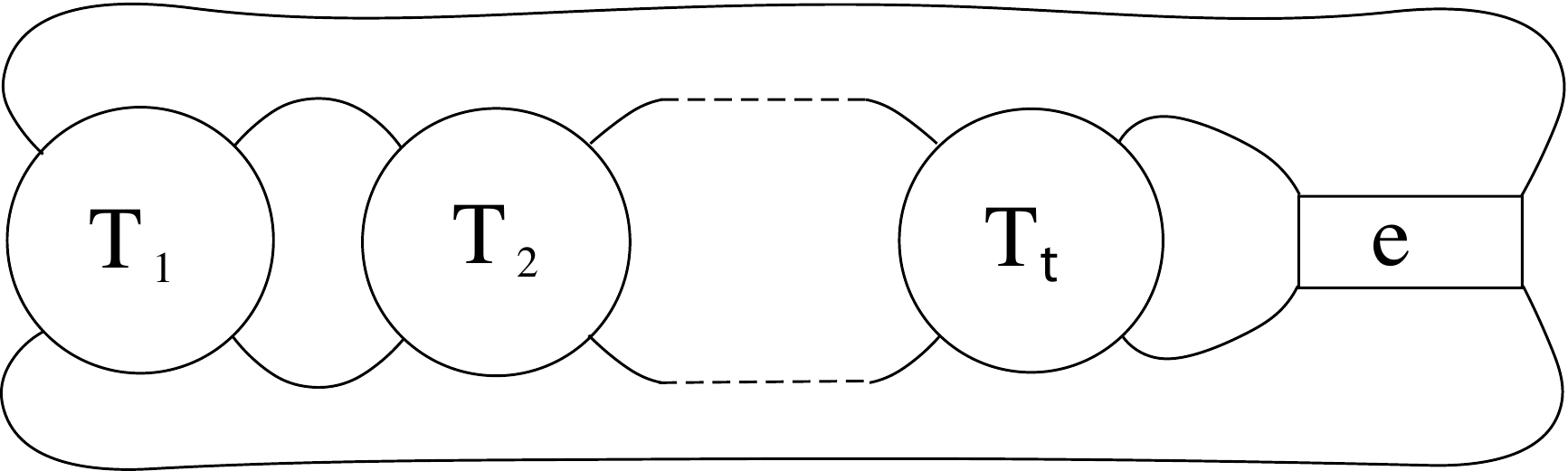}
\caption{An illustration of a Montesinos link diagram $D=K(T_1,
T_2,\ldots, T_t, e)$ where each $T_i$ is a rational tangle. }
\label{montesionsLink}
\end{center}
\end{figure}

\medskip
Let $\frac{\beta_i}{\alpha_i}$ be the rational number whose
continued fraction expansion  is the vector that defines the
rational tangle $T_i$. We will sometimes write $K(T_1,
T_2,\ldots, T_t, e)$ as $K(\frac{\beta_1}{\alpha_1},
\frac{\beta_2}{\alpha_2},\ldots, \frac{\beta_t}{\alpha_t}, e)$.
It is known that a Montesinos link admits a Montesinos diagram
satisfying the following additional condition:
$|\frac{\beta_i}{\alpha_i}|<1$ for each $i$ (hence the continued
fraction of $\frac{\beta_i}{\alpha_i}$ is of the form $(a_{i,1},
a_{i,2}, \cdots,a_{i,n_i},0)$, where $a_{i,j}>0$). See \cite{BZ}
for an explanation of this and the classification of Montesinos
links in general. For a Montesinos link $\mk$, let $D_{\K}$ be a
Montesinos diagram of $\K$ that satisfies this condition and let
$P_i$ be the set of indices $i$ such that $a_{i,j}$ consists of
parallel crossings and $A_i$ be the set of indices $i$ such that
$a_{i,j}$ consists of anti-parallel crossings. We have the
following theorem.

\medskip
\begin{theorem}\label{T2}
Let $\K$ be a Montesinos link with Montesinos diagram
$D_\mk=K(\frac{\beta_1}{\alpha_1},
\frac{\beta_2}{\alpha_2},\ldots, \frac{\beta_t}{\alpha_t}, e)$,
where $|\frac{\beta_i}{\alpha_i}|<1$. Then the number of Seifert
circles in $D_\mk$ is given by the following formula
$$
s(D_{\K})=\left\{
\begin{array}{ll}
\sum_{i=1}^t( \sum_{j\in A_i} (|a_{i,j}|-1)+|P_i|)+2 & {\rm if}\ D_\mk\,{\rm is\ of\ type\ I},\\
\sum_{i=1}^t( \sum_{j\in A_i} (|a_{i,j}|-1)+|P_i|)+|e| + c& {\rm if}\ D_\mk\,{\rm is\ of\ type\ II},
\end{array}
\right.
$$
where $c=2$ if $e=0$ and all tangles end with anti-parallel
vertical twists, and $c=0$ otherwise.
\end{theorem}

\begin{proof}
Consider one of the tangle diagrams
$\frac{\beta_i}{\alpha_i}=(a_{i,1}, a_{i,2},
\cdots,a_{i,n_i},0)$. If the crossings corresponding to $a_{i,j}$
have parallel orientation,  then the crossings correspond to
$a_{i,j-1}$ and $a_{i,j+1}$ must have anti-parallel orientation.
This can be seen as follows: Assume that $a_{i,j}$ represents
$|a_{i,j}|$ vertical twists with a parallel orientation, see
Figure \ref{tangleconfiguration}. Assume further that both
strings are oriented upwards. Then we have two strands entering
the tangle marked by the dashed oval in Figure
\ref{tangleconfiguration} from below. Therefore the other two
strands of that tangle must have an exiting orientation. This
implies that the half-twists at $a_{i,j-1}$ and $a_{i,j+1}$ must
be anti-parallel. A similar argument holds if $a_{i,j}$
represents horizontal twists with a parallel orientation.

\begin{figure}[h!]
\begin{center}
\includegraphics[scale=.6]{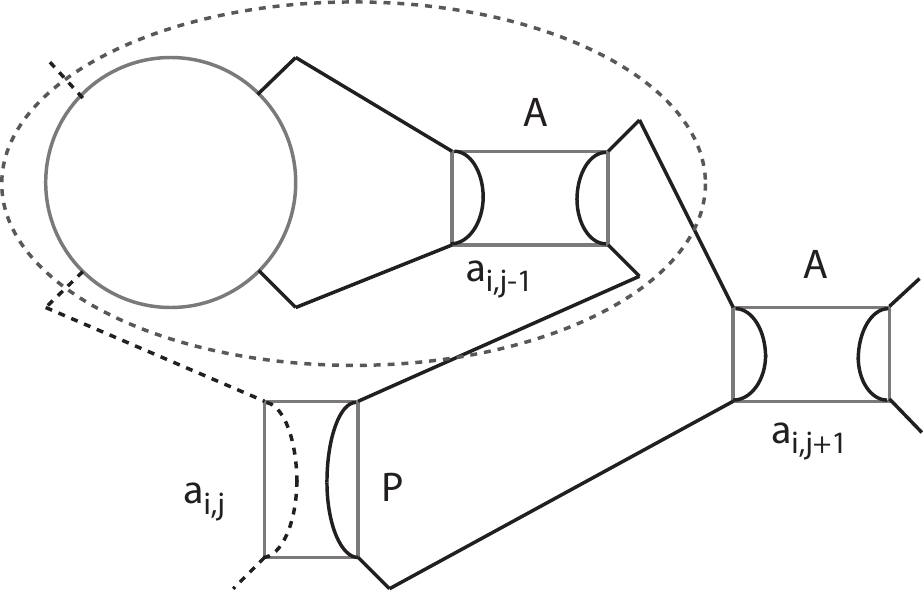}
\caption{ The local structure of parallel and anti-parallel
orientations and  the Seifert circle structure resulting from
nullification under the assumption that $a_{i,j}$ is vertical and
has parallel orientation .} \label{tangleconfiguration}
\end{center}
\end{figure}

\medskip
Moreover, there is a Seifert circle that uses the boxes of all
three entries $a_{i,j-1}$,  $a_{i,j}$ and $a_{i,j+1}$ as shown in
the Figure \ref{tangleconfiguration}. This implies that the
dashed arcs at the top left and bottom left of the tangle marked
by the dashed oval in Figure  \ref{tangleconfiguration} must
belong to the same Seifert circle. It is easy to see that these
properties are the same if $a_{i,j}$ consists of horizontal
twists.

\medskip
Note that after nullification the two arcs in the tangle $T_i$
are changed to a  set of disjoint Seifert circles and two
disjoint arcs connecting two of the four endpoints of the tangle.
The Seifert circles generated be nullification can be grouped
into three different categories. The first group consists of the
{\em small Seifert circles}, namely those generated by two
consecutive half-twists that are anti-parallel. Clearly, if
$|a_{i,j}|>1$ and the corresponding crossings are anti-parallel,
then there are $|a_{i,j}|-1$ such small circles. The second group
consists of the {\em medium sized Seifert circles}, namely the
Seifert circles that are not small but are contained within one
of tangles $T_i$. The two Seifert circles shown in Figure
\ref{tangleconfiguration} are medium sized ones since they are
contained in a tangle and involve parallel crossings. The third
group consists of the rest of the Seifert circles. These are
Seifert circles that involve more than one tangle and are called
{\em large Seifert circles}. Figure  \ref{tangleconfiguration}
also shows that for each $a_{i,j}$ that is parallel, one of the
two arcs after nullification belongs to a medium Seifert circle
and the other belongs to a large Seifert circle. The only
exception occurs when $a_{i,j}$ consists of parallel half twists
and is the last nonzero entry of the tangle, that is $j=n_i$. In
this case both arcs belong to large Seifert circles.

\medskip
We have thus shown the following:

\smallskip\noindent
(i) There are $\sum_{i=1}^t \sum_{j\in A_j} (|a_{i,j}|-1)$ small Seifert circles.

\smallskip\noindent
(ii) There are $-p+\sum_{i=1}^t |P_i|$ medium Seifert circles
where $p$ is the number of  tangles $T_i$ where $a_{i,n_i}$
consists of parallel half twists.

\medskip
If $D_\K$ is of type I, then each $a_{i,n_i}$ consists of
anti-parallel twists and $p=0$.  Furthermore, $e$ consists of
parallel twists as well and it is easy to see that there are
exactly two large Seifert circles. (One passes through all the
tangles at the bottom and the other weaves through all tangles in
a more complex path). This proves the first case of the theorem.

\medskip
Now assume that $D_\K$ is of type II. If a tangle $T_i$ ends with
parallel (anti-parallel)  vertical twists then after
nullification the arc with one end at the NW corner will connect
to the SW corner (NE) corner. If we assume that $e=0$ then we can
see that after nullification there will be exactly $p$ of the
large Seifert circles when $p$ is nonzero and exactly 2 if $p=0$.
If $e\ne 0$ then there will be an additional $|e|-1$ small
circles and the number of large Seifert circles changes by minus
one if $p=0$ and increases by plus one if $p\ne 0$. This proves
the second case of the theorem.
\end{proof}

\begin{corollary}\label{C1}
Let $\K$ be the Montesinos link represented by the diagram
$D_\K=K(\frac{\beta_1}{\alpha_1},
\frac{\beta_2}{\alpha_2},\ldots, \frac{\beta_t}{\alpha_t}, e)$,
where $|\frac{\beta_i}{\alpha_i}|<1$ and $e$ is an integer. Then
we have
$$
n_d({\K})\le\left\{
\begin{array}{ll}
\sum_{i=1}^t( \sum_{j\in P_i} (|a_{i,j}|-1)+|A_i|)+|e|-1& {\rm if}\ D_\mk {\rm is\ of\ type\ I},\\
\sum_{i=1}^t( \sum_{j\in P_i} (|a_{i,j}|-1)+|A_i|) - c+1& {\rm if}\ D_\mk {\rm is\ of\ type\ II},
\end{array}
\right.
$$
where $c=2$ if $e=0$ and all tangles end with anti-parallel
vertical twists, and $c=0$ otherwise. Moreover if $D_\mk$ is
alternating,  then we have equality.
\end{corollary}

\medskip
\begin{proof}
It suffices to prove the statement of equality for an alternating
Montesinos link. In this case the Corollary follows from Theorem
\ref{T2} and the relationship $n_D(\K)=Cr(D)-s(D) +1$ for any
alternating reduced  diagram $D$ of the alternating $\K$.
\end{proof}

\medskip
Since a 4-plat is a Montesinos link that contains only one
rational tangle, we have the following. (Note that a 4-plat is
also  obtained if there are two rational tangles in the
Montesinos link. However this is not important in this context,
see \cite{BZ}.)

\medskip
\begin{corollary}\label{4plats}
Let $\K$ be the 4-plat defined by the vector
$(a_1,a_2,\ldots,a_n)$, $P$ be the set of indices $i$ such that
$a_i$ consists of  parallel crossings and $A$ be the set of
indices $i$ such that $a_i$ consists of anti-parallel crossings.
Then
$$
n_d(\mk)=\sum_{i\in P}(|a_i|-1)+|A|.
$$
\end{corollary}

\begin{proof}
Consider the 4-plat as the Montesinos link given by $D_\K=
K(\frac{\beta_1}{\alpha_1}, e)$, where
$\frac{\beta_1}{\alpha_1}=(a_1,a_2,\ldots,a_{n-1},0)$ and
$e=a_n$. If $a_{n-1}$ is  zero then $\K$ is just an $(e,2)$ torus
link and the statement is true. If $a_{n-1}$ is not zero then we
apply Corollary \ref{C1}. If $D_\K$ is of type I then $e$ is
parallel and in the formula of Corollary \ref{C1} we count $e$ as
parallel and $|e|-1=|a_n|-1$ in the formula of the Corollary. If
$D_\K$ is of type II then $e$ is anti-parallel and will be
counted as the $+1$ in Formula in Corollary \ref{C1}.
\end{proof}

\medskip
\section{The Nullification Numbers and other Link Invariants}

In this section we explore further the relationships among the
three nullification numbers $n(\K)$, $n_r(\K)$ and $n_d(\K)$,  as
well as their relationships with some other link invariants.

\medskip
\subsection{The case of alternating links.} First let us consider the alternating links. We have the following theorem.

\begin{theorem}\label{alterNull}
If $\K$ is an alternating link then we have $n_d(\mk)=n_r(\mk)$.
\end{theorem}

\medskip
\begin{proof}
Since we already have $n_r(\K)\le n_d(\K)$, it suffices to show
that $n_r(\mk)\ge n_d(\mk)$. Let $cr(\K)=n$. Assume  that
$n_r(\mk)$ is obtained by smoothing crossings $c_1$, $c_2$,...,
$c_m$ first in a reduced alternating diagram $D$ of $\K$. This
results in a diagram $D_1$ with $n-m$ crossings that is still
alternating. Assume further that $n_r(\mk)$ is obtained by
deforming $D_1$ to a minimum diagram $D_2$ and some crossings in
it are then smoothed. $D_2$ is necessarily alternating since $D_1$
is alternating and $D_2$ share the same knot type with $D_1$. On
the other hand, $D_1$ can be changed to a reduced alternating
diagram $D_1^\prime$ by performing all possible reduction moves
as shown in Figure \ref{RMonereduction}. $D_1^\prime$ is also
minimum since it is reduced and is alternating. Thus it is flype
equivalent to the diagram $D_2$. Note that for each crossing
reduction move as shown in Figure \ref{RMonereduction}, one
crossing is removed and the number of Seifert circles is reduced
by one at the same time. Thus $m+n_{D_2}=m+Cr(D_2)-s(D_2)
+1=m+Cr(D_1^\prime)-s(D_1^\prime) +1=s+Cr(D_1)-s(D_1)
+1=Cr(D)-s(D) +1=n_d(\mk)$. Therefore no reduction in the number
of nullification steps can be gained by moving to the diagram
$D_2$. Since $D_2$ is still alternating, moving to other minimum
diagrams after smoothing some crossings in it will not result in
a nullification number reduction either by the same argument.
\end{proof}

\begin{figure}[h!]
\begin{center}
\includegraphics[scale=.4]{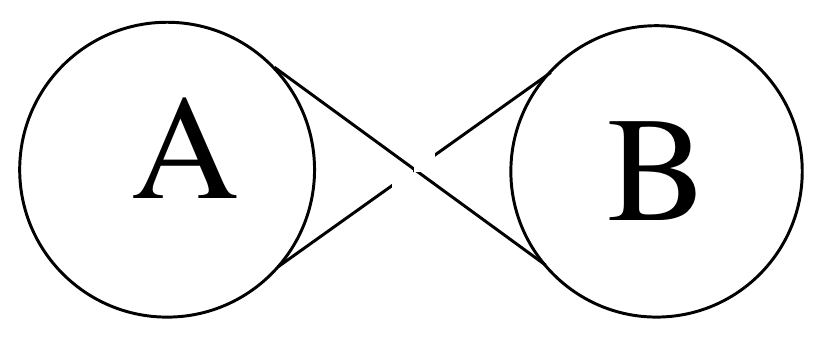}
\caption{A reducible alternating diagram contains a single
crossing that splits the diagram into two parts. Rotation of one
part (either $A$ or $B$) in a proper direction by $180^\circ$
will eliminate one crossing while  preserving the alternating
property of the diagram.} \label{RMonereduction}
\end{center}
\end{figure}

\medskip
On the other hand, even for alternating links, the difference
between $n(\mk)$ and $n_r(\mk)$ can be as large as one wants.

\medskip
\begin{theorem}
For any given positive integer $m$, there exists an alternating
knot $\mk$ such that $n(\K)-n_r(\mk)=n(\K)-n_d(\mk)>m$.
\end{theorem}

\begin{proof} As shown in Figure \ref{nullification4Plat},
the 4-plats of the vector form $(-k,-2,-1,2,k)$ all have general
nullification number one, where $k$ is any positive integer.
\begin{figure}[h!]
\begin{center}
\includegraphics[scale=.4]{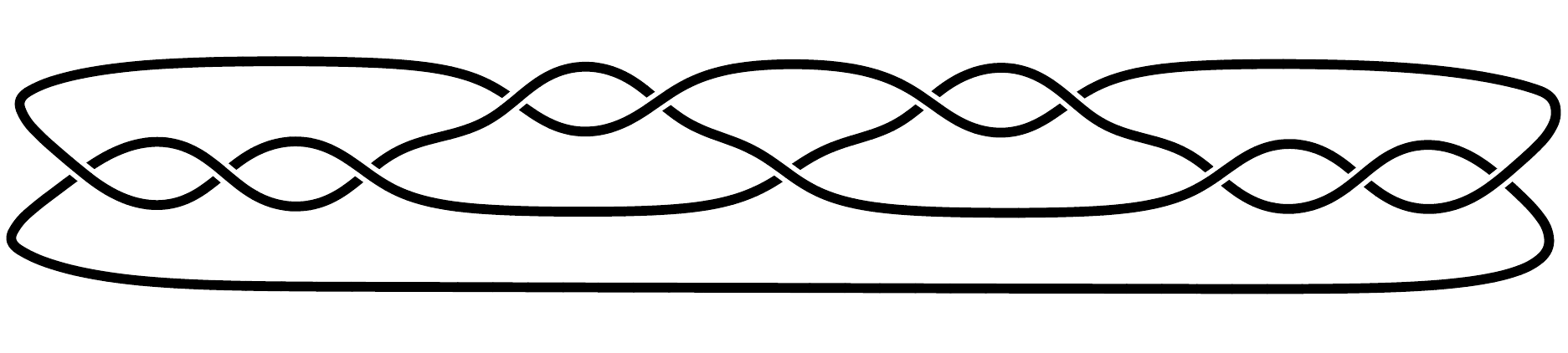}
\caption{\label{nullification4Plat} The 4-plats $(-k,-2,-1,2,k)$
have nullification number one. Here  an example with $k=3$ is
shown. Smoothing the crossing in the center realizes $n(\K)=1$. }
\end{center}
\end{figure}

Notice that $(-k,-2,-1,2,k)$ can be isotoped to $(1,k-1,3,1,k)$
as shown in Figure \ref{4platexample},  which is alternating.
Thus we have an alternating knot $\mk$ with general nullification
number one.  For the minimum diagram of $\mk$ given in Figure
\ref{4platexample} we have, by Corollary \ref{4plats} or equation
\ref{solaeq}, $n_r(\mk)=n_d(\mk)=(2k+4)-5+1=2k$. So the
difference between the diagram nullification number and the
general nullification number is $>m$ if $k>(m+1)/2$.
\end{proof}

\begin{figure}[h!]
\begin{center}
\includegraphics[scale=.4]{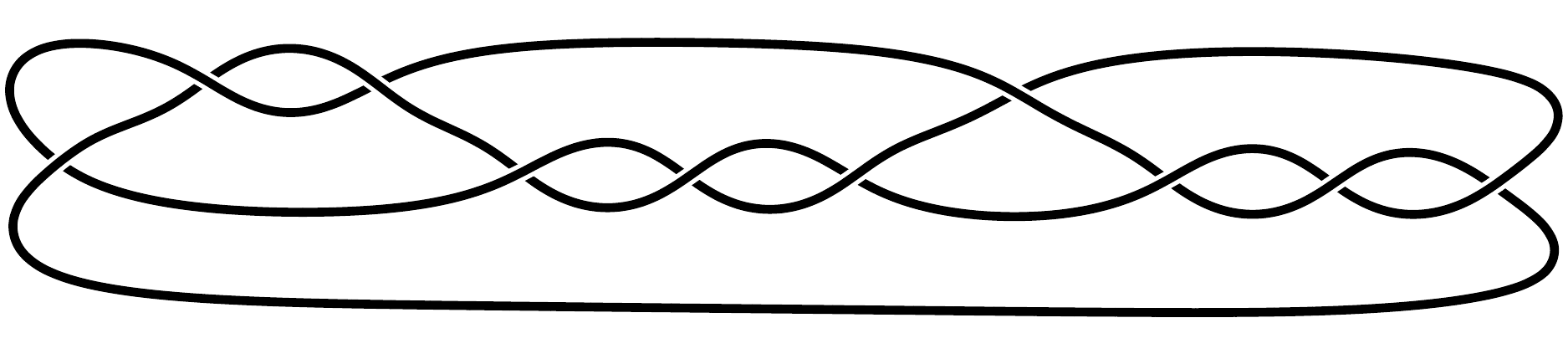}
\caption{\label{4platexample} The 4-plat of Figure
\ref{nullification4Plat}  in a minimum diagram with vector form
$(1,k-1,3,1,k)$ for $k=3$.}
\end{center}
\end{figure}

\medskip
\subsection{The case of non-alternating links.} In order to show that $n_r(\mk)$ and $n_d(\mk)$ are indeed different in general, we need to demonstrate the existence of knots/links $\mk$ such that $n_r(\mk)<n_d(\mk)$. Because of Theorem \ref{alterNull}, such examples can only be found in non-alternating knots and links. Worse, there are no known methods or easy approaches in finding such examples. This subsection is thus devoted to the construction of one single such example.

\medskip
First, let us observe that the knot $8_{20}$ has the following special property. The left of Figure \ref{820isotopy} shows that $n_d(8_{20})=n_r(8_{20})=n_d(8_{20})=1$.
However, after applying a simply isotopy as shown on the right side of Figure \ref{820isotopy}, the new diagram can no longer be nullified by smoothing only one crossing. One has to smooth two crossings.

\begin{figure}[h!]
\begin{center}
\includegraphics[scale=0.8]{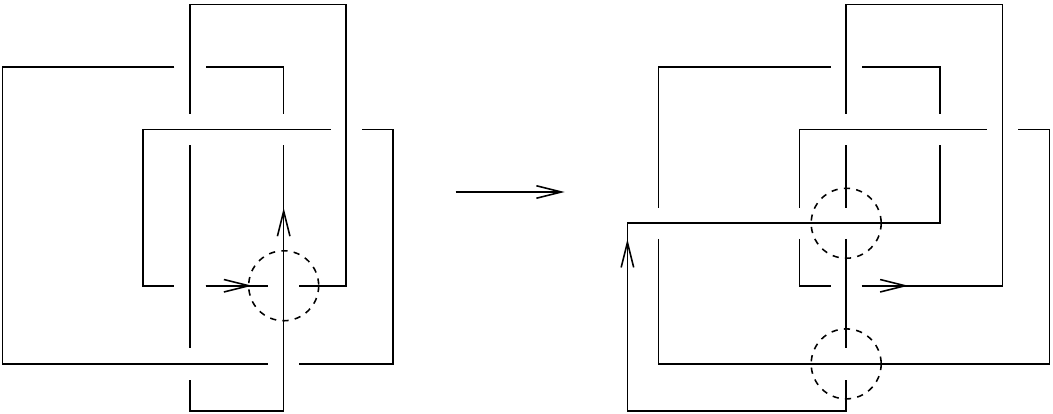}
\caption{\label{820isotopy}Left: A minimum knot projection $M$ of the knot $8_{20}$ that can be nullified by smoothing the crossing circled. Right: A non-minimum diagram $N$ of $8_{20}$ that requires the smoothing of two crossings in order to be nullified (two such crossings are marked with circles).}
\end{center}
\end{figure}

Now we would like to construct an example using this observation.
We construct a three component link $L$ by adding two simple closed curves to $N$ as shown in Figure \ref{speciallink} (drawn by thickened lines). The diagram $D_L$ of
Figure \ref{speciallink} can be shown to be adequate hence is minimum \cite{Crom}. Assign the orientations to the
components as shown in Figure \ref{speciallink}.

\begin{figure}[h!]
\begin{center}
\includegraphics[scale=0.8]{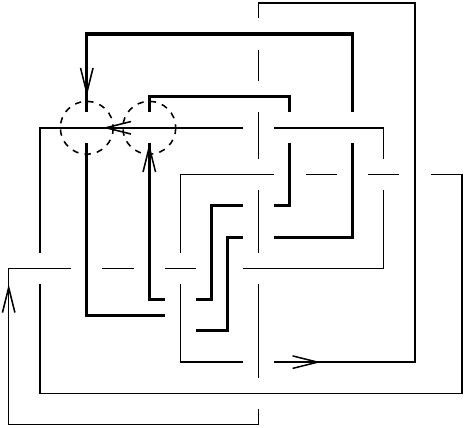}
\caption{\label{speciallink}A 3-component link $L$ with $n_r(L)\le 3$ and $n_d(L)=4$.}
\end{center}
\end{figure}

\medskip
It is easy to see that nullifying the two adjacent crossings marked in Figure \ref{speciallink} allows the
new components to be removed by an ambient isotopy. The resulting diagram is $N$ which can be further isotoped to $M$. Thus by the definition of
$n_{r}$ we have $n_{r}(L)\le 3$.

\medskip
Lacking a more elegant method, we took a programmatic approach for the confirmation that $n_D(L)>3$. First, we observe that the diagram of Figure \ref{speciallink} is the only minimum
diagram of $L$ up to trivial isotopies. This follows from the fact that if we remove either one of the
new components we obtain an alternating  and hence minimum
diagram that admits no flypes. Thus any diagram of such a two
component link has to look like the one shown. The two additional
components are ``parallel'' and therefore the only change we can
make is to exchange them. This however does not change the
diagram. We then implemented the nullification procedure which yielded all knot/link diagrams obtained by all possible combinations of 3 or less crossing smoothing steps on $D_L$. Using a Gauss code modification program and the Mathematica\copyright  KnotTheory package's Jones polynomial
computation we were able to verify that none of  these resulted in a knot/link with
the polynomial of a trivial knot/link. Since $n_d(L)\le 4$ by our construction of $L$, we have
shown that $n_d(L)=4$. Thus we have shown an example of a three component link $L$ with $n_r(L)\le 3$ and $n_d(L)=4$.

\medskip
\subsection{The genus, unknotting number and the signature vs $n(\mk)$.}

\medskip
There is a simple inequality between the nullification number and the unknotting number:

\begin{lemma}\label{unknot}
Let $\mk$ be any knot then $n(\mk)\le 2 u(\mk)$, where $u$ denotes the unknotting number of a knot.
\end{lemma}

{\bf Proof:}
It suffices to show that a strand passage can be realized by two nullification moves. This is shown in Figure \ref{strandpassage}.

\begin{figure}[h!]
\begin{center}
\includegraphics[scale=.35]{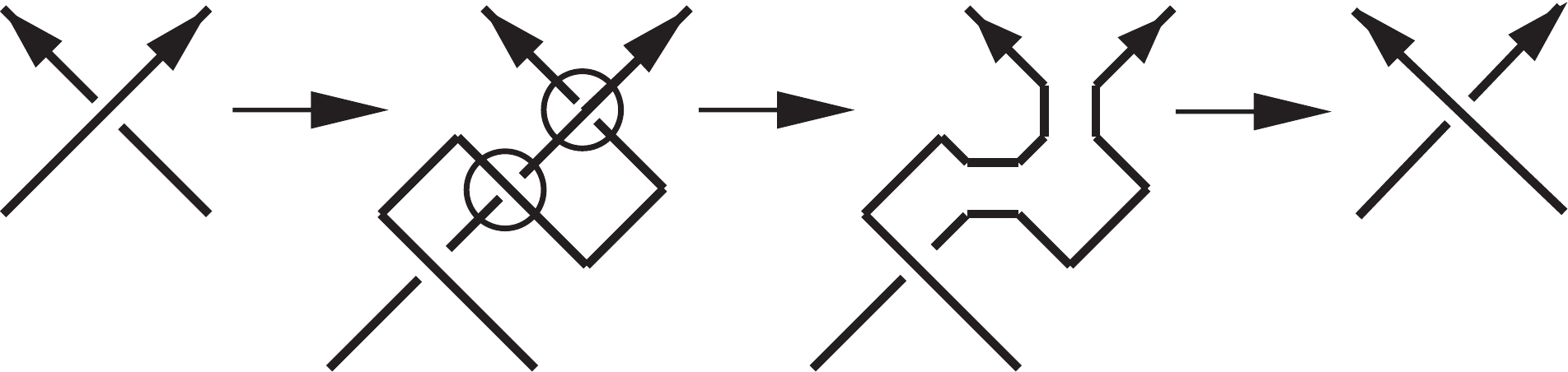}
\caption{\label{strandpassage} Nullifying two crossings is equivalent to a strandpassage.}
\end{center}
\end{figure}

\medskip
The next theorem shows that an inequality as given in Lemma \ref{unknot} does not exist the other way around, that is the unknotting number can not be bound from above by a multiple of the general nullification number.

\medskip
\begin{theorem}
For any given positive integer $m$, (1) there exists an
alternating knot $\mk$  such that $g(\K)-n(\K)>m$ where $g(\K)$
is the genus of $\mk$ and (2) there exists a link $\mk$ such that
$u(\mk)-n(\mk)>m$. In other words, the general nullification
number does not impose a general upper bound on the genus and the unknotting number of a link.
\end{theorem}

\medskip
\begin{proof}
Notice that the genus of the knot $\mk$ in Figure
\ref{4platexample} is $k$.  Thus if $k>m+1$, then
$g(\mk)-n(\mk)=k-1>m$. For the second part of the theorem,
consider the torus link $T(3n,3)$ ($n$ is an arbitrary positive
integer) as shown in Figure \ref{toruslink}, in which two of the
three components are oriented in parallel (say clockwise) and the
third component is oriented in the other direction (say
counterclockwise). By nullifying any one of the crossings between
two components with  opposite orientations in Figure
\ref{toruslink} we obtain the unlink. Thus $n(T(3n,3))=1$. On the
other hand, by the Bennequin Conjecture \cite{Ben, Wei}, the
unknotting number of $T(3n,3)$ is $3n$ and the result of the
second part of the theorem follows. \end{proof}

\begin{figure}[h!]
\begin{center}
\includegraphics[scale=.2]{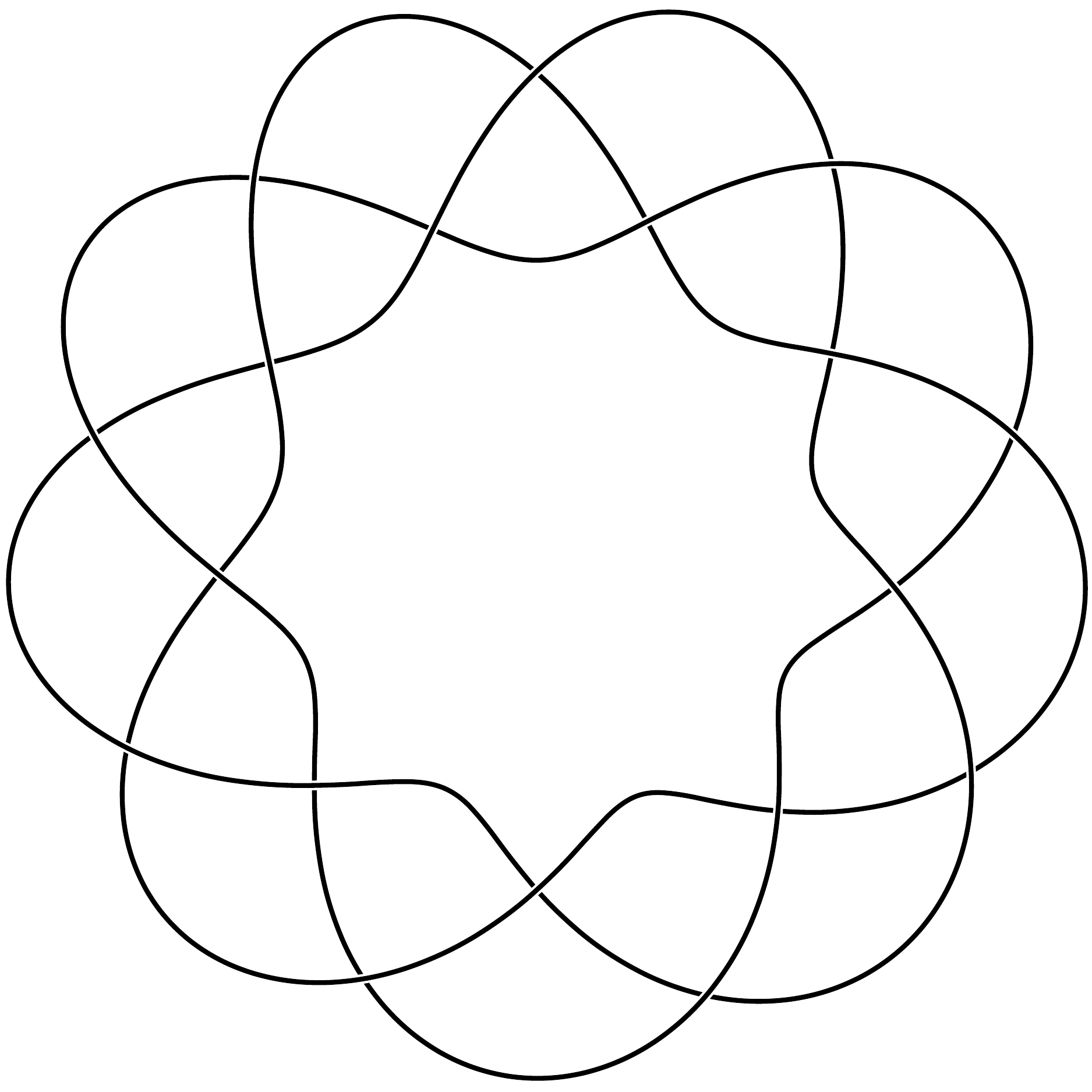}
\caption{\label{toruslink}The torus link $T(3n,3)$ with
nullification number  one and unbounded unknotting number. The
special case $n=3$ is shown here.}
\end{center}
\end{figure}

\medskip
The signature $\sigma(\K)$ of a link $\K$ is defined as the
signature of the matrix  $M+M^T$ where $M$ is the Seifert matrix
obtained from any regular diagram of $\K$. The definition and
computation of a Seifert matrix is beyond the scope of this paper
and we refer the reader to any standard text in knot theory such
as \cite{BZ, Crom}.

\medskip
\begin{theorem}\label{signature}
For any oriented link $\K$ we have $n(\K)\ge |\sigma(\K)|$.
\end{theorem}

\begin{proof}
We use an approach similar to the one used in the proof of
Theorem 6.8.2 in \cite{Crom}  (where the relationship between the
signature and the unknotting number is investigated). We will
need the following lemma \ref{sigmaseries} from \cite{Jones}.  A
$\sigma$-series of an $n\times n$ matrix of rank $r$ is a
sequence of submatrices $\Delta_i$ such that (i) $\Delta_i$ is an
$i\times i$ matrix; (ii) $\Delta_i$ is obtained from
$\Delta_{i+1}$ by removing a single row and a single column; and
(iii) no two consecutive matrices $\Delta_i$ and $\Delta_{i+1}$
are singular when $i<r$.

\medskip
\begin{lemma}\label{sigmaseries}\cite{Jones}
Let $A$ be a symmetric $n\times n$ matrix with a $\sigma$-series
as described above.  Put $\Delta_i=1$. Then the signature of $A$
is given by
$$
\sigma(A)= \sum^n_{i=1} {\rm sgn}(\det(\delta_{i-1})\det( \delta_i)),
$$
where ${\rm sgn}()$ is the sign function.
\end{lemma}

Let $D_+$, $D_-$ and $D_0$ be three link diagrams that are
identical except at one  crossing as shown in Figure
\ref{skeinrelation} and let $L_+$, $L_-$ and $L_0$ be their
corresponding link types, we claim that
$$
|\sigma(L_\pm)-\sigma(L_0)|\le 1.
$$

\begin{figure}[h!]
\begin{center}
\includegraphics[scale=.5]{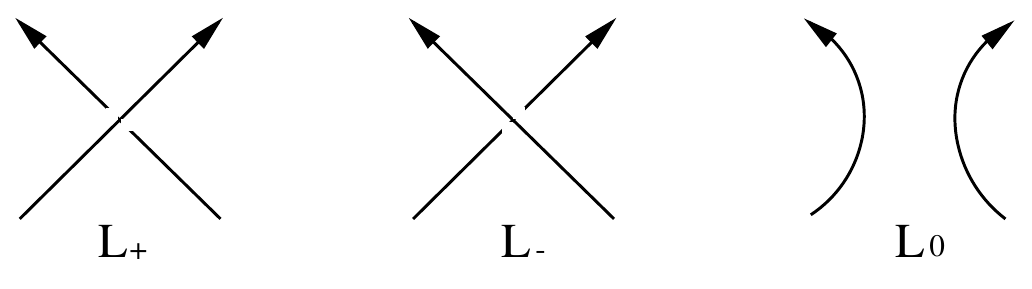}
\caption{\label{skeinrelation} Three diagrams that differ only at one crossing.}
\end{center}
\end{figure}

\medskip
To prove this, let $F_+$, $F_-$ and $F_0$ be the projection
surfaces constructed from  $D_+$, $D_-$ and $D_0$, respectively.
Let $M_+$, $M_-$ and $M_0$ be the Seifert matrices constructed
from these surfaces. If $F_0$ is disconnected then $L_+$ and
$L_-$ are isotopic links equivalent to a connected sum $L_1\#L_2$
for some links $L_1$ and $L_2$. Therefore in this case $\sigma
(L_1\#L_2)=\sigma(L_1)+\sigma(L_2)=\sigma(L_1\sqcup
L_2)=\sigma(L_0)$ and
$\sigma(L_+)-\sigma(L_0)=\sigma(L_-)-\sigma(L_0)=0$.

\medskip
If $F_0$ is connected then $F_+$ and $F_-$ are obtained from
$F_0$  by adding a twisted  rectangle at the crossing that is
switched or eliminated. Therefore the Seifert matrices $M_+$ and
$M_-$ have one additional column and row added to the Seifert
matrix $M_0$ as shown below for $M_+$ where $b$ is an additional
loop that passed through the added rectangle and back through the
rest of the surface $F_0$ and $(a_1,\cdots,a_n)$ is a basis for
$H_1(F_0)$.

 \begin{center}
 \begin{tabular}{c|cccc}
 &$a_1^+$&$\cdots$&  $a_n^+$&$b^+$ \\ \hline
 $a_1$&&&&$v_1$\\
 $\vdots$&&$M_0$&&$\vdots$\\
 $a_n$&&&&$v_1$\\
 $b$&$\lambda_1^+$&$\cdots$&  $\lambda_n^+$&$\beta$
  \end{tabular}
 \end{center}
We write $A_*=M_*+M_*^T$ where $*=+,-$ or $0$. The matrices $A_*$
may not be singular,  however they are non-singular if the
corresponding $L_*$ is a knot. In either case we can use Lemma
\ref{sigmaseries} to obtain that:
$$
\sigma(A_\pm)-\sigma(A_0)= sign(\det(A_\pm)\det(A_0))= \delta,
$$
where $*=+$ or $-$ and from this it follows that $\delta=\pm 1$ or $0$.

\medskip
This result implies that in any nullification sequence of $\K$,
the smoothing of a crossing  can only change the signature by at
most one. Since the signature of the trivial link is zero and
$n(\mk)$ is the minimum number of smoothing moves needed to
change $\K$ to a trivial link, it follows that $|\sigma(\K)|\le
n(\K)$.
\end{proof}

\medskip
\begin{corollary}\label{corunknot}
Let $\mk$ be any knot then $|\sigma(\mk)|\le n(\mk)\le 2 u(\mk)$. In particular if $|\sigma(\mk)| = 2 u(\mk)$, then  $|\sigma(\mk)|=n(\mk)=2 u(\mk)$.
\end{corollary}

\medskip
Note that the above Corollary is quite powerful if one wants to determine the actual general nullification number for small knots in the knot table. In \cite{EMS} the general nullification number of all but two knots of up to 9 crossings is determined.
The following corollary is an immediate consequence of Theorem
\ref{signature} and the fact  that all knots have even signatures
\cite{M}.

\begin{corollary}
If $\K$ is a knot and $n(\K)=1$, then $\sigma(\K)=0$.
\end{corollary}

In \cite{Hon} the signature of an oriented alternating link $\mk$
is related to the writhe  via the concept of nullification
writhe.  Let $D$ be a reduced alternating diagram of $\mk$ and in
this diagram we have a nullification sequence of crossings $c_1$,
$c_2$, ... , $c_k$, where $k=n_D=n_d(\mk)$, then the
nullification writhe introduced in \cite{C} is defined as the sum
of the signs of the crossings $c_1$, $c_2$, ... , $c_k$ and is
denoted by $Wr(n_D)$. We have:

\medskip
\begin{theorem}\cite{Hon}\label{nullwrithe}
For an alternating oriented link
$\sigma(\mk)+Wr(n_D)=0$.
\end{theorem}

\medskip
The example of the 4-plat given in Figures
\ref{nullification4Plat}  and \ref{4platexample}  (with $k=3$)
represents the knot $10_{22}$. This knot has $\sigma(10_{22})=0$
(In fact we can see that a nullification sequence for the diagram
in Figure  \ref{4platexample} has a nullification sequence of
three positive and three negative crossings.) On the other hand
we see that the inequality in Theorem \ref{signature} is strict
since $n(10_{22})=1$. Note that the signature of a 4-plat can be
computed by the following explicit formula.

\begin{theorem}\label{signature4}\cite{Sto3}
Let $\K$ be a 4-plat knot then
$$
\sigma(\K)=\sum_{i=1}^{2g} (-1)^{i-1} sign(a_i),
$$
where the vector $(a_{1},a_{2},...,a_{2g})$ is a continued
fraction expansion of $\K$ using  only even integers. (Such a
continued fraction expansion is called an {\em even continued
fraction expansion}.)
\end{theorem}

\medskip
In fact, the length of the even continued fraction expansion is $2g$ where $g$ is the genus of $\K$ \cite{Crom}.
We also have the following two corollaries:

\medskip
\begin{corollary}\label{nDeqn}
If $D$ is a reduced alternating diagram of an alternating
oriented link $\mk$ such that all  crossings in a nullification
sequence of $D$ have the same sign, i.e., $|Wr(n_D)|=n_D$, then
$n_d(\mk)=n(\mk)$.
\end{corollary}

\begin{proof}
We have $|\sigma(\mk)|=|Wr(n_D)|=n_D=n_d(\mk)\ge n(\mk)\ge
|\sigma(\mk)|$. Thus the equality holds.
\end{proof}

\medskip
\begin{corollary}\label{nullnumberk}
For any even positive integer $k$ the number of knots $\K$ with
$Cr(\mk)$ crossings and  $|n(\K)|\ge k$ grows exponentially with
$\sqrt{Cr(\mk)}$.
\end{corollary}

\begin{proof}
Let us assume for simplicity that $m=Cr(\mk)$ is odd and consider
the 4-plats $\mk$ with a  vector form $(a_{1},a_{2},...,a_{2g})$
such that $a_i$ is even and  $a_ia_{i+1}<0$ for all $i$. Then $\sigma(\mk)=2g$. The
crossing number $Cr(\mk)$ is given by \cite{Crom}
$$
Cr(\K)=(\sum_{i=1}^{2g}|a_i|)-2g+1.
$$
If $m=Cr(\K)$ is sufficiently large then we can generate such a
vector by partitioning the positive integer $m+2g-1$ into $2g$
even integers $b_i$ where $b_i\ge 2$. Here we set
$a_i=(-1)^{i+1}\cdot b_i$. This is equivalent to partitioning the
integer $m-2g-1$  into $2g$ even integers $c_i$ where $c_i\ge 0$
which in turn is equivalent to the number of partitions of the
integer $(m-1)/2-g$ into $2g$ integers $d_i$ where $d_i\ge 0$. A
standard result in number theory about the number of partitions
then implies the result.
\end{proof}

\medskip
\section{Nullification Number One Links}

\medskip
In this section we explore the special family of links whose
general nullification number is one.  The first question is about
the number of such links. In the following we will give a partial
answer to this question. Let $(a_{1},a_{2},\ldots,a_{k})$ be the
standard vector form of a rational link $\K$ corresponding to the
rational number $\frac{p}{q}$. Consider the rational link
$\K^\prime$ defined by the vector
$(a_{1},a_{2},\ldots,a_{k},\varepsilon,-a_{k},\ldots,-a_{1})$,
where $\varepsilon=\pm1$. Let $\frac{a}{b}$ be the fraction for
$\K^\prime$, then we can get
$\frac{a}{b}=\frac{p^{2}}{\left(\varepsilon+pq\right)}$ by using
a method from \cite{Sieb}. Furthermore, when $\K$ is a knot,
$\K^\prime$ is also a knot and has general nullification number
1.

\medskip
Two fractions $\frac{p^{2}}{\left(\varepsilon_{1}+pq\right)}$ and
$\frac{p^{2}}{\left(\varepsilon_{2}+pq^\p\right)}$ represent the
same knot iff
$\varepsilon_{1}+pq\equiv\varepsilon_{2}+pq^\p\left({\rm mod}\,
p^{2}\right)$ or
$\left(\varepsilon_{1}+pq\right)\left(\varepsilon_{2}+pq^\p\right)\equiv1\left({\rm mod}\,
p^{2}\right)$. If $\varepsilon_{1}=\varepsilon_{2}$ then this
yields $pq\equiv pq^\p\left({\rm mod}\, p^{2}\right)$ or
$pq\equiv-pq^\p\left({\rm mod}\, p^{2}\right)$. So $q\equiv q^\p\left({\rm mod}\,
p\right)$ or $q\equiv-q^\p\left({\rm mod}\, p\right)$. This implies that
different fractions $\frac{p}{q}$ result  in different 4-plats
$\frac{p^{2}}{\left(\varepsilon_{1}+pq\right)}$ (up to mirror
images).

\medskip
Thus we have shown that for each rational knot $\K$, there exists
a rational knot $\K^\prime$  (unique up to mirror image) with
nullification number one. Furthermore,
$Cr(\K^\prime)=2\sum^k_{i=1} a_i=2Cr(\K)$. Since the number of
rational knots grows exponentially, we have shown the following
theorem.

\medskip
\begin{theorem}\label{nullnumber1}
The number of knots $\K$ with $n(\K)=1$ and $Cr(\K)\le n$
grows exponentially in terms of $n$.
\end{theorem}

\medskip
We would like to point out that the rational knots considered
above do not contain all  nullification number one rational
knots. We give two additional example of 4-plat knot families with nullification number one.
In the two examples we do not follow the usual vector notation for rational knots such as in \cite{BZ}. Instead we adopt the sign assignment convention as shown in Figure \ref{convention} for the crossings that are in the boxes in Figures  \ref{family1} and  \ref{family2}. Here the crossings have two ends marked as on the bottom and the other two ends as on the top and therefore the crossings in Figure \ref{convention} can not be rotated by $90$ degrees.

\begin{figure}[h!]
\begin{center}
\includegraphics[scale=0.8]{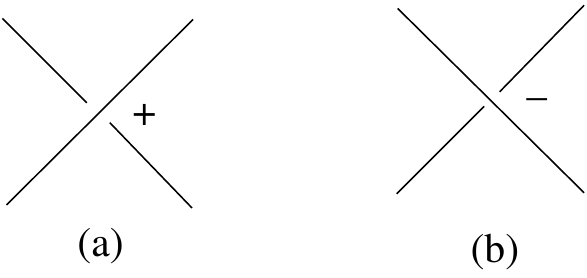}
\caption{\label{convention}Crossing (a) is positive and crossing (b) is negative. Notice that this is not how the sign of a crossing in a rational knot is assigned and is only used for the examples here.}
\end{center}
\end{figure}

\medskip
\begin{example}{\em  Consider the rational knot defined by a vector of the form $(-2a,-2,-2b,2,2a,2b)$, where the first $2a$ means a sequence of $2a$ positive crossings between the first and second strings using the sign convention given in Figure \ref{convention}, and so on, as shown in the first diagram of Figure \ref{family1}.
Note that the actual sign convention of the crossings in the boxes does not matter, as long as boxes with opposite signs have twists that are mirror images of each other.
By a rotation involving the two top boxes in the first diagram, followed by a proper flype involving the resulting two boxes and some other isotopes, one can see that the first diagram is equivalent to the second one. From there it is relatively easy to see the second diagram can be isotoped to the third diagram. If we smooth the crossing as marked in the third diagram, we end up with the fourth diagram. It is not too hard to see that the fourth diagram is the trivial knot. Thus this family of rational knots is also of general nullification number one. The details of the isotopies used are left to the reader as an exercise.

\begin{figure}[h!]
\begin{center}
\includegraphics[scale=1.0]{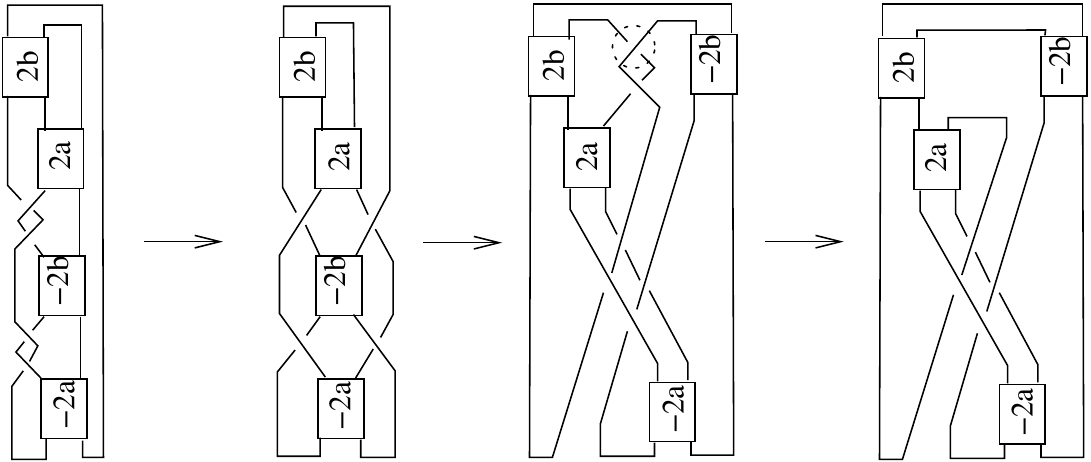}
\caption{\label{family1}The family of rational knots with the vector form $(-2a,-2,-2b,2,2a,2b)$ using the sign convention as defined in Figure \ref{convention}. These knots
have general nullification number one, where the signed number on a box indicates the number of half twists with the shown sign.}
\end{center}
\end{figure}
}
\end{example}

\begin{example}{\em
The second example (of other nullification number one rational knots)
is similarly constructed. Here the knot family consists of rational
knots defined by vectors of the form $(-2b,-2,-2a,2b,2,2a)$ as shown
in the first diagram of Figure \ref{family2}. The isotopes leading to
the single nullification crossing of the knot are illustrated in the rest of the figure.

\begin{figure}[h!]
\begin{center}
\includegraphics[scale=1.0]{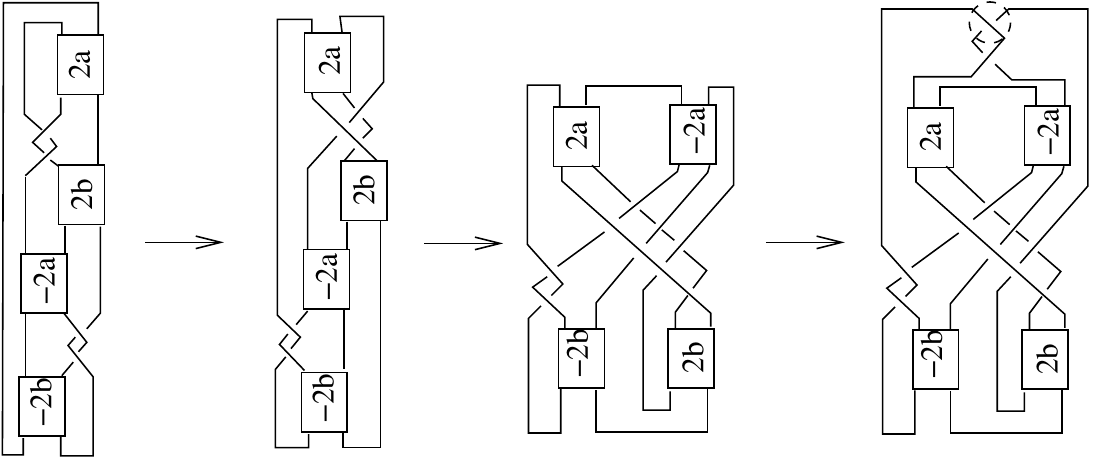}
\caption{\label{family2}Another family of rational knots (defined
by vectors of form $(-2b,-2,-2a,2b,2,2a)$ using the sign convention
 as defined in Figure \ref{convention}) with general nullification number one.}
\end{center}
\end{figure}
}
\end{example}
\medskip
Note that the actual sign convention as shown in in Figure \ref{convention} does not matter as long as the the crossings in the two boxes with the same label have opposite handedness.
These three families of rational knots have one interesting
property in common, namely that they are all {\rm ribbon knots}
(to be defined next). In fact it was conjectured in \cite{Cass}
that these were the only rational ribbon knots. This conjecture
was recently proven in \cite{Lisca}.

\medskip
Let $K$ be a link in $S^3$, and $b:I\times I\rightarrow S^{3}$ be
an imbedding  such that $b\left(I\times I\right)\cap
K=b\left(I\times\partial I\right)$. Let $K_{b}=\{K\setminus
b(I\times\partial I)\}\cup b\left(\partial I\times I\right)$. If
$K_b$ is a link with an orientation compatible with $K$ then $K_b$
is called a \textit{banding} of $K$, or is obtained from $K$ by band
surgery (along $b$). It turns out that nullification and banding
are equivalent operations. See Figure \ref{nullBandEquiv}.

\begin{figure}[h!]
\begin{center}
\includegraphics[angle=270,scale=.3]{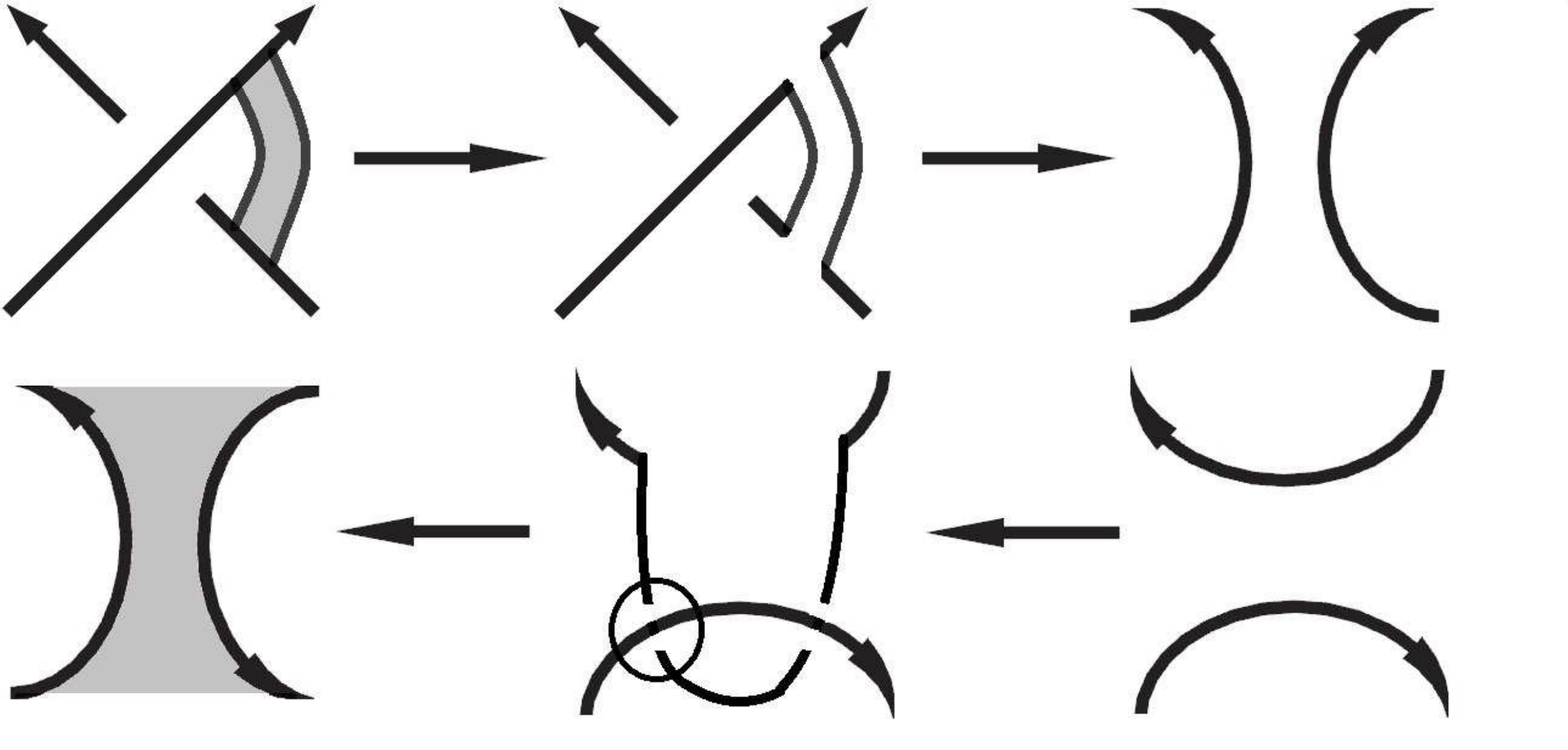}
\caption{\label{nullBandEquiv}Top: Nullification via banding. Bottom: Banding via nullification.}
\end{center}
\end{figure}

\medskip
A knot $\mk$ is a \textit{ribbon knot} if it is a knot obtained
from a trivial $(m+1)$-component link by band surgery along $m$
bands for some $m$. The minimum of such number $m$ is called the
\textit{ribbon-fusion number} of $\mk$ and is denoted by ${\rm
rf}(\mk)$. (This concept was introduced by Kanenobu
\cite{kanBand}.) If a knot can be nullified in one step then it
yields the trivial 2-component link. So a knot has nullification
number one if and only if it is a 1-fusion ribbon knot. Using
this, one can easily find all nullification number one knots.
This can be done by quickly referencing all ribbon presentations
of ribbon knots with 10 or fewer crossings (such as in
\cite{Kawa}). All such diagrams either explicitly have the
necessary crossing, or will have one after applying a
Reidemeister move of type II as shown in Figure \ref{ribbonNull}.
Furthermore, the fact that 1-fusion
ribbon knots have nullification number one, together with
the following theorem due to Tanaka \cite{tanBrid}, implies
that the nullification number and the bridge number of a knot are unrelated in general.

\medskip
\begin{theorem}\cite{tanBrid}
For any pair $(p,q)$ of positive integers with $1 \leq p \leq q$,
there exists a  family of infinitely many composite ribbon knots
{$K_i$} such that $2(p+q)+1 \leq b(K_i) \leq 2(p+q+1)$ and
$rf(K_i)=p$, where $b(K_i)$ denotes the bridge number of the knot $K_i$.
\end{theorem}

In particular if we let $(p,q)=(1,q)$ then we get an infinite
family of 1-fusion ribbon knots  (nullification number one knots)
with arbitrarily large bridge number.

\begin{figure}[h!]
\begin{center}
\includegraphics[scale=0.5]{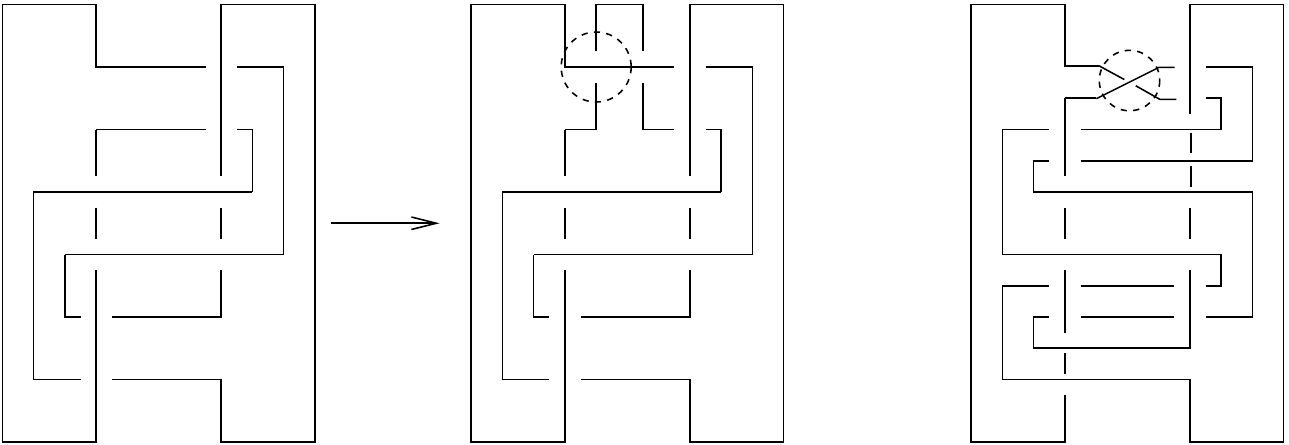}
\caption{\label{ribbonNull}Two examples of nullifying ribbon presentations:
on the left is the knot $6_1$ where the presentation requires a type II
Reidemeister move first, on the right is the knot $8_9$ where no
Reidemeister moves are needed.}
\end{center}
\end{figure}

\medskip
For higher fusion numbers it is not \textit{a priori} clear
whether nullification number is equal to the fusion number.
Fusion number $m$ relies on the separation of a ribbon knot into
$(m+1)$ trivial components, but nullification does not have such
a restriction. So for a given ribbon knot $\K$, it is only
obvious that $n(\K) \leq {\rm rf}(\mk)$.

\medskip
Let us end our paper with a few open questions.

\medskip
1. For a knot or link $L$, how big the difference between $n_r(L)$ and $n_d(L)$ can be? Can we find a class of knots/links such that $n_d(L)-n_r(L)$ is unbounded over all $L$ from this family?

\medskip
2. If $D^\p$ is a diagram obtained from an alternating (reduced) diagram $D$ by one crossing change (so $D^\p$ is no longer alternating and it may even be the trivial knot/link), how much smaller is $n_{D^\prime}$ compared to $n_D$? For alternating knots with unknotting number one, $n_{D^\prime}$ is simply zero hence this difference can be as large as one wants. However, is there a way to relate this problem with the unknotting numbers in general?

\medskip
3. By nullifying one crossing in a diagram $D$ (not necessarily minimum) of a knot/link $\mk$, we obtain a new knot/link. How many different knots/links can be obtained this way? In particular, is this number bounded above by $Cr(\mk)$?

\end{document}